\documentclass{amsart}
\usepackage{amsfonts, amsmath, amssymb, amscd, latexsym, graphicx}

\usepackage{tikz}
\usetikzlibrary{arrows,backgrounds,decorations,automata}

\newtheorem{thm}{Theorem}[section]
\newtheorem{prop}[thm]{Proposition}
\newtheorem{lem}[thm]{Lemma}

\newcommand{\Groves}{MR1425318}
\newcommand{\FredenKS}{MR2777003}
\newcommand{\FredenK}{MR2328173}
\newcommand{\FredenA}{FredenA}

\newcommand{\DiekLaun}{MR2787455}
\newcommand{\LaunDiss}{LaunDiss}
\newcommand{\Brazil}{MR1169911}
\newcommand{\CollinsEG}{MR1249578}
\newcommand{\EdjvetJ}{MR1151287}
\newcommand{\LS}{MR0577064}

\newcommand{\ElderLinearBS}{MR2776987}
\newcommand{\ElderCFBS}{MR2142503}

\newcommand{\Tom}{TomWong}
\newcommand{\CGA}{CGA}
\newcommand{\KKM}{KKM}

\renewcommand\>{\rangle}

\newcommand{\bs}[2]{BS(#1,#2)}

\begin{document}

\title{Metric properties of Baumslag--Solitar groups}

\author{Jos\'e Burillo}

\address{Departament de Matem\`atica Aplicada IV, EETAC-UPC, C/Esteve Torrades 5,
08860 Castelldefels, Barcelona, Spain} \email{burillo@ma4.upc.edu}

\author{Murray Elder}
\address{School of Mathematical \& Physical Sciences, The~University~of~Newcastle, Callaghan, New South Wales, Australia}
\email{murray.elder@newcastle.edu.au}

\thanks{Research supported by MEC grant MTM2011--25955, and Australian Research Council (ARC) grant  FT110100178}

\begin{abstract}
We compute estimates for the word metric of  Baumslag--Solitar groups in terms of the Britton's lemma normal form. As a corollary, we find lower bounds for the growth rate for the groups $\bs pq$, with $1<p\le q$.
\end{abstract}

\maketitle

\section{Introduction}

In this article we investigate the word length of elements in the groups
 $\bs pq$,  presented by
$$
\<a,t\,|\,ta^pt^{-1}=a^q\>.
$$
We use the Britton normal form to obtain an estimate for the word length, and use it to  compute a lower bound for the growth rate.

Recall that a function $f:G\to \mathbb R$ is a  {\em metric estimate} for a group $G$ with finite symmetric generating set $S$ if there exist constants $C_1,D_1,C_2,D_2>0$ so that
for every element $g\in G$, we have
$$
C_1f(x)-D_1\leq ||x||_S\leq C_2f(x)+D_2,
$$
where $||\cdot ||_S$ is the word metric with respect to $S$.

 Collins, Edjvet and Gill  and independently Brazil \cite{\Brazil, \CollinsEG}  showed that the groups $\bs 1q$ have  rational growth series, with explicit closed form series given in   \cite{\CollinsEG}. Edjvet and Johnson also gave
  rational growth series for the groups $\bs qq$  \cite{\EdjvetJ} (see also \cite{\LaunDiss} Section 2.8).
 Finding an expression for the growth series or the growth rate of $\bs pq$ for $1<p< q$
has proven to be a stubbornly difficult problem.
Freden {\em et al.} have made some progress in a series of papers \cite{\FredenA, \FredenK, \FredenKS}.  Wong
has also made some progress  \cite{\Tom}, finding various estimates for the growth rate. In each case the authors have  focussed on the (difficult) problem of computing the growth of just  elements equal to a power of the generator $a$, the so-called {\em horocyclic} elements. In  \cite{\FredenKS} Freden and Knudson prove that  the growth series of the horocyclic subgroup is rational when $p\mid q$, and conjecture that when $p\nmid q$ it is not D-finite.

Computing geodesics in these groups is an equally non-trivial problem. The second author gave a linear time algorithm for the case $\bs 1q$ \cite{\ElderLinearBS}, and Diekert and Laun gave a quadratic time algorithm for
the case $\bs pq$ when $p\mid q$ \cite{\DiekLaun, \LaunDiss}.

There has also been interest in the {\em language} of geodesics for these groups.
Groves showed that no set of geodesics surjecting to $\bs 1q$  can be regular  \cite{\Groves}. The second author constructed  a context-free and 1-counter language of geodesics for $\bs 12$
\cite{\ElderCFBS}.  Freden and  Adams give a context-sensitive combing in the case of $\bs 27$ \cite{\FredenA}.

 The paper is organised as follows. In Section \ref{sec:solv} we compute the metric estimate for $\bs 1q$, which we extend to the general case in Section \ref{sec:nonsolv}. In Sections \ref{sec:bound}--\ref{sec:autom} we use the estimate to compute lower bounds for the growth rate. In Section \ref{sec:exactbounds} we compare the bounds obtained with some known exact values for the growth rate.

We wish to thank Eric Freden,  Antoine Gournay and  Alexey Talambutsa for helpful comments and improvements to the paper.

\section{The groups $\bs 1q$}\label{sec:solv}

We are assuming that $q>1$. The group $\bs 1q$ admits the presentation
$$
\<a,t\,|\,tat^{-1}=a^q\>,
$$
so observe that by rewriting
$$
ta=a^qt\qquad ta^{-1}=a^{-q}t\qquad at^{-1}=t^{-1}a^{q}\qquad a^{-1}t^{-1}=t^{-1}a^{-q}
$$
we see that every element admits an expression of the type
$$
t^{-m}a^{N}t^n
$$
with $m,n\geq 0$, and $N$ can only be multiple of $q$ if one of the $m,n$ are zero. Under these conditions, it is easy to see that this expression is unique.

From here we can find the expression of the estimate of the metric.

\begin{prop}\label{bound1q} There exist constants $C_1,C_2,D_1,D_2>0$ such that for every element $x=t^{-m}a^{N}t^n$ of $\bs 1q$,  $N\neq 0$, we have
$$
C_1(m+n+\log |N|)-D_1\leq ||x||\leq C_2(m+n+\log |N|)+D_2,
$$
where $||x||$ is the word metric with respect to the generators $a,t$.
\end{prop}
Observe that the base of the logarithm is irrelevant, since a change of base would only imply an adjustment of the constants, so we will work with the base that is most convenient in each case.  Also observe that when $N=0$ the normal form word $t^k$ for $k\in\mathbb Z$ is a geodesic.

\begin{proof} To prove the upper bound, 
assume (taking the inverse if not) that $N>0$, and write $N$ in base $q$ as
$$
N=\sum_{i=0}^rk_iq^i
$$
with $0\leq k_i<q$ and $k_r\ne 0$. Observe that $r=\lfloor\log_q N\rfloor$, and that
\begin{equation}\label{normal1q}
t^{-m}a^{N}t^n=t^{-m}\left(a^{k_0}ta^{k_1}ta^{k_2}\ldots ta^{k_r}t^{-r-1}\right)t^n
\end{equation}
which has length at most $m+n+2q(r+1)$, which gives the desired inequality (with $C_2=D_2=2q$).

For the lower bound, let $b_1\dots b_k$ be a geodesic for $x$. Define a  sequence of elements $x_{k+1}, x_k,\dots, x_1$ by $x_{k+1}=x$ and
$x_i=xb_k^{-1}\dots b_i^{-1}$ for $1\leq i\leq k$.
Let $m_i,n_i\in \mathbb N, N_i\in\mathbb Z$  such that $x_i$ has normal form
 $t^{-m_i}a^{N_i}t^{n_i}$.
 Since $x_1$ is the identity, the integer sequence $\{m_i+n_i\}$ must go from $m+n$ to 0 in $k$ steps. Since multiplying $t^{-m}a^Nt^n$ by a generator reduces the sum $m+n$ by at most 1,
 we must have that $m+n\leq k$.
 The sequence $ |N_i|$ also goes to 0, with multiplication of $x_i$ by a generator changing $|N_i|$ as follows.
\begin{itemize}
\item For $x_it$, we have that $|N_{i}|=|N_{i-1}|$.
\item For $x_{i}t^{-1}$, we  have that $|N_{i}|=|N_{i-1}|$, except in the  case $n_{i}=0$. In this case $x_{i}=t^{-m_{i}}a^{N_{i}}$ and $x_{i}t^{-1}=t^{-m_{i}-1}a^{qN_{i}}$, and $$\log |N_{i-1}|=\log |qN_{i}|=\log|N_{i}|+\log q>\log |N_{i}|.$$
\item For $x_{i+1}a^{\pm 1}$ we have
 $$x_{i+1}a^{\pm 1}=t^{-m_{i+1}}a^{N_{i+1}}t^{n_{i+1}}a^{\pm 1}=t^{-m_{i+1}}a^{N_{i+1}}a^{\pm q^{n_{i+1}}}t^{n_{i+1}}$$
 so
 $N_{i+1}$ can be reduced by subtracting at most $q^{n_{i+1}}$.
\end{itemize}
It follows that only multiplication by $a^{\pm 1}$ can reduce $|N_i|$.
In the worst case (to maximize $|N|$)
each generator reduces $|N|$ by $q^{n_i}$ where  $n_i\leq n+k\leq 2k$ so $|N|\leq kq^{2k}$ so $$\log |N|\leq \log k+2k\log q\leq k+2k\log q=(2\log q+1)k.$$ 

Putting this together we have $$m+n+\log |N|\leq k+(2\log q+1)k=(2\log q+2)k$$ so $C_1=\frac1{2(\log q+1)}$ and $D_1=0$.

\end{proof}

We remark that the result also follows indirectly from recent work of Kharlampovich, Khoussainov and Miasnikov \cite{\KKM} who show that $BS(1,q)$ is {\em graph automatic}, which implies the normal form associated to the graph automatic structure gives a metric estimate for the group (see Lemma~2.15 in \cite{\CGA}).

\section{The groups $\bs pq$}\label{sec:nonsolv}

In this section we prove metric estimates for the groups $\bs pq$ for $1\leq p\leq q$. We make use of the following normal form for group elements,  based on Britton's lemma (see \cite{\LS}).

\begin{lem} Any element $x$ of $\bs pq$ for $1\leq p\leq q$ can be written  uniquely as $x=w(a,t)a^N$ where
$$
w(a,t)\in\{t,at,a^2t,\ldots,a^{q-1}t,t^{-1},at^{-1},a^2t^{-1},\ldots,a^{p-1}t^{-1}\}^*
$$ and is freely reduced.
\end{lem}

The proof is straightforward by performing the following rewritings, so that the only large power of $a$ appearing is on the right hand side:
\begin{itemize}\item
remove canceling pairs $aa^{-1},a^{-1}a,tt^{-1},t^{-1}t$;
\item replace $a^{r\pm dq}t$ by $a^rta^{\pm dp}$, where $0\leq r<q$;
\item replace $a^{s\pm dp}t^{-1}$ by $a^st^{-1}a^{\pm dq}$, where $0\leq s<p$.
\end{itemize}
 Uniqueness is an easy exercise based on Britton's lemma. The word $w(a,t)$ is of the form
$$
t^{m_0}a^{r_1}t^{m_1}a^{r_2}t^{m_2}\ldots a^{r_k}t^{m_k}
$$
where:
\begin{itemize}
\item $m_i,r_i\in\mathbb Z$,
\item $m_k=0$ only if the word $w$ is empty,
\item $m_i\neq 0$ for $i\geq 1$,
\item $0<r_i<q$,
\item if $p\le r_i<q$, then $m_i>0$.
\end{itemize}

We  now use this normal form to obtain metric estimates. We study first the case $p<q$.

\begin{thm} \label{boundpq}There exist constants $C_1, C_2,D_1,D_2>0$ such that for every element $x\in\bs pq$  for $1\leq p < q$  written as $w(a,t)a^N$, we have
$$
C_1(|w|+\log (|N|+1))-D_1\leq ||x||\leq C_2(|w|+\log (|N|+1))+D_2.
$$
\end{thm}

\begin{proof}
We first prove the upper bound. If $N=0$ we are done.
Assume (taking the inverse of $a^N$ if not) that $N>0$ and
 write $N=d_1q+r_1$ with $0\leq r_1<q$. Then $a^N=a^{r_1}ta^{d_1p}t^{-1}$.
 Note that  $$d_1p=d_1q\left(\frac{p}{q}\right)\leq \left(d_1q+r_1\right)\frac{p}{q} = N\left(\frac{p}{q}\right).$$
 Now write $d_1p=d_2q+r_2$ with $0\leq r_2<q$, where
 $$d_2p=d_2q\frac{p}{q}\leq \left(d_2q+r_2\right)\frac{p}{q} =d_1p\frac{p}{q}\leq N\left(\frac{p}{q}\right)^2.$$
Repeat to obtain $$a^N=a^{r_1}ta^{r_2}t\dots a^{r_k}ta^{d_kp}t^{-k}$$ with $1\leq d_kp<q$ when the process terminates, and observe that $d_kp$ cannot be zero, or else the process terminates in the previous step.  We have  $1\leq d_kp\leq N\left(\frac{p}{q}\right)^k$, so we deduce that $k\leq \log_{q/p} N$.  Our word has length
at most $qk+q+k$ since each $r_i<q$ and $d_1p<q$.
It follows that the length obtained for the word $a^N$ is at most
$$(q+1)\log_{q/p} N + q$$ which yields our upper bound.

Next,  the lower bound.
Let $x_1x_2\dots x_{n}$ be a geodesic for $x$  with $n=||x||$, and let $w_ia^{N_i}$ be the normal form for the prefix of length $i$. We have $w_0=\epsilon$ and $N_0=0$.

If $x_{i+1}=a^{\pm 1}$ then $w_{i+1}=w_i$, and $|N_{i+1}|\leq |N_i|+1$.
If $x_{i+1}=t^{-1}$, put $N_i=dp+r$ with $0\leq r<p$ and $d$ an integer.
\begin{itemize}\item If $r=0$ and $w_i$ ends with $t$, write $w=ua^ct$ with $0\leq c<q$ and $u$ empty or ending in $t^{\pm 1}$.
Then $w_ia^{N_i}t^{-1}=ua^cta^{dp}t^{-1}=ua^{c+dq}$. It follows that $|w_{i+1}|=|u|<|w_i|$ and
$$|N_{i+1}|\leq |N_i|\left(\frac{q}{p}\right)+q.$$

\item Otherwise $w_ia^rt^{-1}$ is freely reduced. In this case $w_ia^{N_i}t^{-1}=w_ia^rt^{-1}a^{dq}$ so $|w_{i+1}|\leq |w_i|+p$ and $$|N_{i+1}|\leq |N_i|\left(\frac{q}{p}\right).$$
\end{itemize}

If $x_{i+1}=t$, put $N_i=dq+s$ with $0\leq s<q$ and $d$ an integer.
\begin{itemize}\item If $s=0$ and $w_i$ ends with $t^{-1}$, write $w=ua^ct^{-1}$ with $0\leq c<p$ and $u$ empty or ending in $t^{\pm 1}$.
Then $w_ia^{N_i}t=ua^ct^{-1}a^{dq}t=ua^{c+dp}$. It follows that $|w_{i+1}|=|u|<|w_i|$ and
$$|N_{i+1}|\leq |N_i|+p <  |N_i|\left(\frac{q}{p}\right)+q.$$

\item Otherwise $w_ia^st$ is freely reduced. In this case $w_ia^{N_i}t=w_ia^sta^{dp}$ so $|w_{i+1}|\leq |w_i|+q$ and $$|N_{i+1}|\leq |N_i|\left(\frac{p}{q}\right).$$
\end{itemize}

It follows that for  multiplication by any generator we have
$$|w_{i+1}|\leq |w_i|+q \hspace{1cm} \mathrm{and} \hspace{1cm} |N_{i+1}|\leq  |N_i|\left(\frac{q}{p}\right)+q.$$

After $n$ multiplications the value of $|w_{n}|$ can be at most $qn$, while $|N_{n}|$ is bounded as follows.
We have
$ N_0  =  0$,
 $ |N_1|  \leq  q$,
 $ |N_2| \leq  \left(\frac{q}{p}\right)q+q$,
 $ |N_3|  \leq   \left(\frac{q}{p}\right)^2q+ \left(\frac{q}{p}\right)q+q$ and so on,
so $$|N_{n}|\leq \sum_{i=0}^{n-1} \left(\frac{q}{p}\right)^iq=q\frac{\left(\frac{q}{p}\right)^{n}-1}{\left(\frac{q}{p}\right)-1}< C\left(\frac{q}{p}\right)^{n}$$   where $C=\frac{q}{\left(\frac{q}{p}\right)-1} >1$ as $qp>q-p$.
We then have
$$|N_{n}|+1\leq  C\left(\frac{q}{p}\right)^{n}+1\leq 2C\left(\frac{q}{p}\right)^{n}$$ since $C>1$ and $\frac{q}{p}>1$.
Then $\log_{q/p} \left(|N_{n}|+1\right) \leq \log_{q/p} (2C)+ n=\log_{q/p} (2C)+ ||x||$.

The two  lower bounds combine to give the result with $D_1=\log_{q/p} (2C)$ and $C_1=\frac{1}{q+1}$.
\end{proof}

The case $p=q$ is considerably easier.

\begin{lem} \label{boundpp}There exists a constant $C_1>0$ such that for every element $x\in\bs pp$  for $p\geq 1$  written as $w(a,t)a^N$, we have
$$
C_1(|w|+|N|)\leq ||x||\leq |w|+|N|.
$$
\end{lem}
\begin{proof}
Since $|w|+|N|$ is the word length of the normal form, the upper bound is immediate.

The lower bound follows the same argument as the $p<q$ case.
Let $x_1\dots x_{n}$ be a geodesic for $x$ with $n=||x||$, and define $w_i,N_i$  as before.

Multiplication by $a^{\pm 1}$ gives $|w_{i+1}|=|w_i|$ and $|N_{i+1}|\leq |N_i|+1$.

For multiplication by $t^{\pm 1}$, let $N_i=dp+r$ for $0\leq r<p$ and $d$ an integer.\begin{itemize}\item  If $r=0$ and $w_i$ ends in $t^{\mp 1}$ we have $w_i=ua^ct^{\mp 1}$ (with $c<p$) and
$w_ia^{N_i}t^{\pm 1}=ua^ct^{\mp 1}a^{dp}t^{\pm 1}=ua^{c+dp}$
so $|w_{i+1}|\leq |w_i|$ and $|N_{i+1}|\leq |N_i|+p$.
\item
Otherwise we have $w_ia^{N_i}t^{\pm 1}=w_ia^rt^{\pm 1}a^{dp}$ so $|w_{i+1}|\leq |w_i|+p$ and $|N_{i+1}|\leq |N_i|$.
\end{itemize}
Then $|w_{n}|\leq p||x||$ and $|N_{n}|\leq p||x||$ which gives our lower bound with $C_1=\frac1{2p}$.
\end{proof}

\section{Lower bound for the growth rate for $\bs pq$}\label{sec:bound}

For $q> p>1$, the exact growth rate for $\bs pq$ is not known. Here, we will make use of the bounds specified above to find lower bounds for these rates. The growth function is given by
$$
\gamma(n)=\#B(n)=\#\{x\in\bs pq\,:\,||x||\leq n\},
$$
but observe that we can consider the alternate set using the upper bound given above:
$$
D(n)=\{x=w(a,t)a^N\in\bs pq\,:\,|w|+(q+1)\log_{q/p}N+q\leq n\}
$$
and the bound implies precisely that $D(n)\subset B(n)$. Hence, we have that $\#D(n)$ is a lower bound for $\gamma(n)$.

To estimate the number of elements with normal form $w(a,t)a^N$ which satisfy
$$
|w|+(q+1)\log_{q/p}N+q=k,
$$
observe that if $(q+1)\log_{q/p}N+q=k$, then the number $N$ is of the order of an exponential with base
$$
\left(\frac qp\right)^\frac 1{q+1}
$$
which goes to 1 as $q$ grows. The conclusion one can deduce from this is that in the set
$$
\{x=w(a,t)a^N\,:\,|w|+(q+1)\log_{q/p}N+q=k\}
$$
the dominant part will be the part of those elements satisfying $|w|=k$ because it will be an exponential with base larger than
$$
\left(\frac qp\right)^\frac 1{q+1}
$$
at least asymptotically.

Now consider the set
$$
E(n)=\{w(a,t)\,:\,|w|\leq n\}.
$$
Observe that if an element can be written as a word $w(a,t)$, then its length is bounded above by $|w|$, so we have that $E(n)\subset B(n)$. Hence the cardinality $\#E(n)$ is a genuine lower bound. And note also that this lower bound works as well for the case $BS(p,p)$.

Observe that the language used in the normal form, i.e.
$$
\{t,at,a^2t,\ldots,a^{q-1}t,t^{-1},at^{-1},a^2t^{-1},\ldots,a^{p-1}t^{-1}\}^*
$$
is regular, so its elements can be described as the words accepted by a finite state automaton. Since the elements are normal forms and Britton's lemma ensures that different normal forms will give different elements, the set $E(n)$ is precisely the set of words accepted by this finite state automaton with length at most $n$. The number of these elements can be asymptotically estimated using the eigenvalues of the adjacency matrix for the automaton, which gives the
  rough lower bounds for the growth rate of the $\bs pq$ groups as shown in Table~\ref{tableA}. In the next section we will obtain better bounds by choosing different normal forms, and we give more details on how we obtain bounds from the automata there.

\begin{center}
\begin{table}[h!]
\begin{tabular}{c|ccccccccc}
&$q=2$&3&4&5&\ldots&10&\ldots&20\\
\hline
$p=2$&2&2.14790&2.20557&2.22919&\ldots&2.24668&\ldots&2.24698\\
3&&2.26953&2.31651&2.33529&\ldots&2.34841&\ldots&2.34859\\
4&&&2.35930&2.37627&\ldots&2.38786&\ldots&2.38801\\
5&&&&2.39246&\ldots&2.40345&\ldots&2.40358\\
\ldots&&&&&&\ldots&\ldots&\ldots\\
10&&&&&&2.41396&\ldots&2.41409\\
\ldots&&&&&&&&\ldots\\
20&&&&&&&&2.41421\\
\end{tabular}
\caption{Lower bounds for the growth rate of $\bs pq$ groups obtained from the metric estimate. \label{tableA}}
\end{table}
\end{center}

\section{Improving the lower bound}\label{sec:autom}

The normal forms obtained at the beginning of section \ref{sec:nonsolv} are by no means the only ones possible. A normal form which will produce shorter normal forms, and hence better lower bounds, is given in the following lemma.

\begin{lem} Any element $x$ of $\bs pq$ for $1\leq p\leq q$ can be written  uniquely as $x=w(a,t)a^N$ where
$$
\begin{array}{rl}
w(a,t)\in\{&\hspace{-3mm}t,at,a^2t,\ldots,a^{\alpha}t,\\
&\hspace{-3mm}a^{-1}t,a^{-2}t,\ldots,a^{-\beta}t,\\
&\hspace{-3mm}t^{-1},at^{-1},a^2t^{-1},\ldots,a^{\gamma}t^{-1},\\
&\hspace{-3mm}t^{-1},a^{-1}t^{-1},a^{-2}t^{-1},\ldots,a^{-\delta}t^{-1}\}^*
\end{array}
$$ and is freely reduced. The exponents are
$$
\alpha=\left\lfloor\frac q2\right\rfloor\qquad
\beta=\left\lfloor\frac {q-1}2\right\rfloor\qquad
\gamma=\left\lfloor\frac p2\right\rfloor\qquad
\delta=\left\lfloor\frac {p-1}2\right\rfloor
$$
\end{lem}

For clarity, the exponents are given by the following table:
\begin{center}
\begin{tabular}{l|l|c|c|c|c}
&&$\alpha$&$\beta$&$\gamma$&$\delta$\\
\hline
$p=2k+1$&$q=2\ell+1$&$\ell$&$\ell$&$k$&$k$\\
$p=2k+1$&$q=2\ell$&$\ell$&$\ell-1$&$k$&$k$\\
$p=2k$&$q=2\ell+1$&$\ell$&$\ell$&$k$&$k-1$\\
$p=2k$&$q=2\ell$&$\ell$&$\ell-1$&$k$&$k-1$
\end{tabular}
\end{center}
and finally observe that since $p\leq q$, we have that $\gamma\leq\alpha$ and $\delta\leq\beta$.

The proof is straightforward by performing the following rewritings, so that the only large power of $a$ appearing is on the right hand side:
\begin{itemize}\item
remove canceling pairs $aa^{-1},a^{-1}a,tt^{-1},t^{-1}t$;
\item replace $a^{r\pm dq}t$ by $a^rta^{\pm dp}$, where $-\beta\leq r\leq\alpha$;
\item replace $a^{s\pm dp}t^{-1}$ by $a^st^{-1}a^{\pm dq}$, where $-\delta\leq s\leq\gamma$.
\end{itemize}
Uniqueness is an easy exercise based on Britton's lemma. The word $w(a,t)$ is of the form
$$
t^{m_0}a^{r_1}t^{m_1}a^{r_2}t^{m_2}\ldots a^{r_k}t^{m_k}
$$
where:
\begin{itemize}
\item $m_i,r_i\in\mathbb Z$,
\item $m_k=0$ only if the word $w$ is empty,
\item $m_i\neq 0$ for $i\geq 1$,
\item $-\beta\leq r_i\leq\alpha$,
\item if $r_i<-\delta$ or $r_i>\gamma$, then $m_i>0$.
\end{itemize}

The words will then be accepted by a finite state automaton. As an example, the case for $\bs 23$ has words in
$$
\{t, at,  a^{-1}t,  t^{-1},  at^{-1}\}^*
$$ and freely reduced, 
which are accepted by the automaton in Figure \ref{fig:BS23}
where the accept states are $S$, 1 and 2. Note that:
\begin{itemize}
\item The word starts at the state $S$ with any letter $a$, $a^{-1}$, $t$ or $t^{-1}$.
\item The word is in state 1 if the last letter was a $t$, and the next letters allowed are $a$, $a^{-1}$ or $t$.
\item The word is in state 2 if the last letter was a $t^{-1}$, and the next letters allowed are $a$, $a^{-1}$ or $t^{-1}$.
\item The word is in state 3 if the last letter was an $a$ but the last two were not $a^2$, and then the next letters allowed are $a$, $t$ or $t^{-1}$.
\item The word is in state 4 if the last letter was $a^{-1}$, and the next letter allowed is only $t$.
\end{itemize}

It is a standard procedure
(either by computing the dominant eigenvalue of the adjacency matrix, or by writing down a regular grammar for the language of the automaton and applying the Chomsky--Sch\"utzenberger theorem)
 to compute the asymptotics of lengths of words accepted by this automaton. We obtain:

\begin{prop}\label{prop:22} The growth rate for $\bs 23$ is bounded below by
$$
\frac{1+\sqrt{13}}{2}=2.30278\ldots .
$$
\end{prop}

\begin{center}
\begin{figure}
\begin{tikzpicture}[scale=.77, ->,>=stealth',shorten >=1pt,auto,node distance=2.5cm,scale=1.3]
\tikzstyle{every state}=[fill=white,draw=black,text=black]

    \node  [state, accepting] (S) at (4,4) {$S$};
   \node  [state, accepting] (p) at (1,4) {$1$};
      \node  [state, accepting] (m) at (7,4) {$2$};
   \node  [state] (a1) at (4,7) {$3$};
  \node  [state] (a2) at (4,1) {$4$};

   \path (S) edge [below] node {$t$} (p);
   \path (S) edge [ below] node {$t^{-1}$} (m);
   \path (S) edge [left] node {$a$} (a1);
 \path (S) edge [right] node {$a^{-1}$} (a2);

      \path (p) edge [loop left] node {$t$} (p);

   \path (m) edge [loop right] node {$t^{-1}$} (m);

   \path (m) edge [left] node {$a$} (a1);
   \path (p) edge [right] node {$a$} (a1);

   \path (a1) edge [bend right, above] node {$t$} (p);
   \path (a1) edge [bend left] node {$t^{-1}$} (m);
   \path (a2) edge [bend left, left] node {$t$} (p);

   \path (p) edge [right] node {$a^{-1}$} (a2);
   \path (m) edge [right] node {$a^{-1}$} (a2);

    \end{tikzpicture}
 \caption{The automaton for the words $w$ in $\bs 23$}

   \label{fig:BS23}
\end{figure}
\end{center}

This construction readily extends to the general case. See Figure \ref{fig:BS47} for an example, the automaton for $BS(4,7)$, which gives a bound of 2.85502. By analyzing the automaton we can compute the lower bounds for the growth rates for all groups.

\begin{center}
\begin{figure}
\begin{tikzpicture}[scale=.77, ->,>=stealth',shorten >=1pt,auto,node distance=2.5cm,scale=1.3]
\tikzstyle{every state}=[fill=white,draw=black,text=black]

    \node  [state, accepting] (S) at (4,10) {$S$};
   \node  [state, accepting] (p) at (1,10) {$1$};
      \node  [state, accepting] (m) at (7,10) {$2$};
   \node  [state] (a1) at (4,13) {$3$};
  \node  [state] (a2) at (4,7) {$4$};
  \node  [state] (a3) at (4,16) {$5$};
  \node  [state] (a4) at (4,4) {$6$};
  \node  [state] (a5) at (4,19) {$7$};
  \node  [state] (a6) at (4,1) {$8$};

   \path (S) edge [below] node {$t$} (p);
   \path (S) edge [ below] node {$t^{-1}$} (m);
   \path (S) edge [left] node {$a$} (a1);
 \path (S) edge [right] node {$a^{-1}$} (a2);
 \path (a1) edge [left] node {$a$} (a3);
 \path (a3) edge [left] node {$a$} (a5);

\path (a2) edge [right] node {$a^{-1}$} (a4);
 \path (a4) edge [right] node {$a^{-1}$} (a6);

      \path (p) edge [loop left] node {$t$} (p);

   \path (m) edge [loop right] node {$t^{-1}$} (m);

   \path (m) edge [left] node {$a$} (a1);
   \path (p) edge [right] node {$a$} (a1);

   \path (a1) edge [bend right, above] node {$t$} (p);
   \path (a1) edge [bend left] node {$t^{-1}$} (m);
   \path (a2) edge [bend left, left] node {$t$} (p);
   \path (a2) edge [bend right, right] node {$t^{-1}$} (m);

   \path (p) edge [right] node {$a^{-1}$} (a2);
   \path (m) edge [right] node {$a^{-1}$} (a2);

   \path (a3) edge [bend right, left] node {$t$} (p);
   \path (a5) edge [bend right, left] node {$t$} (p);
   \path (a4) edge [bend left, right] node {$t$} (p);
   \path (a6) edge [bend left, right] node {$t$} (p);

   \path (a3) edge [bend left, right] node {$t^{-1}$} (m);

    \end{tikzpicture}
 \caption{The automaton for the words $w$ in $\bs 47$}

   \label{fig:BS47}
\end{figure}
\end{center}

\begin{thm}\label{thm:growth_rate} The growth rate for the Baumslag--Solitar group $BS(p,q)$ for the case  $4\leq p\leq q$ is bounded below by the largest zero of the polynomial
$$
P_{pq}(x)=x^{\ell+1}-x^{\ell}-2(x^{\ell-1}+x^{\ell-2}+\ldots+x^{\ell-k+1})-C_kx^{\ell-k}-2(x^{\ell-1}+x^{\ell-2}+\ldots+x)-C_\ell
$$
where
$$
C_k=\left\{
\begin{array}{ll}
1&\text{if }p=2k\\
2&\text{if }p=2k+1
\end{array}
\right.
\qquad
C_\ell=\left\{
\begin{array}{ll}
1&\text{if }q=2\ell\\
2&\text{if }q=2\ell+1
\end{array}
\right.
$$
The remaining cases for $2\leq p\leq 3$ are:
$$
\begin{array}l
P_{22}(x)=x^2-x-2\\
P_{23}(x)=x^2-x-3\\
P_{2q}(x)=x^{\ell+1}-x^{\ell}-x^{\ell-1}-2(x^{\ell-1}+x^{\ell-2}+\ldots+x)-C_\ell,\text{ for }q\geq4\\
P_{33}(x)=x^2-x-4\\
P_{3q}(x)=x^{\ell+1}-x^{\ell}-2x^{\ell-1}-2(x^{\ell-1}+x^{\ell-2}+\ldots+x)-C_\ell,\text{ for }q\geq4
\end{array}
$$
with the same $C_{\ell}$ as above.
\end{thm}

Table~\ref{tableB} contains some bounds for growth rates for $\bs pq$ using this method, which are significantly improved compared to  those obtained in the previous section.
\begin{center}\begin{table}[h!]
\begin{tabular}{c|cccccccc}
&$q=2$&3&4&5&\ldots&10&\ldots&20\\
\hline
$p=2$&2&2.3028&2.4142&2.5115&\ldots&2.6083&\ldots&2.6180\\
3&&2.5616&2.6511&2.7321&\ldots&2.8071&\ldots&2.8136\\
4&&&2.7321&2.8063&\ldots&2.8739&\ldots&2.8794\\
5&&&&2.8751&\ldots&2.9365&\ldots&2.9413\\
\ldots&&&&&&\ldots&\ldots&\ldots\\
10&&&&&&2.9917&\ldots&2.9952\\
\ldots&&&&&&&&\ldots\\
20&&&&&&&&2.999966
\end{tabular}\caption{Improved lower bounds for the growth rate of $\bs pq$ groups. \label{tableB}}
\end{table}
\end{center}

\section{Some upper bounds and some exact values}\label{sec:exactbounds}
In his Master's thesis, Tom Wong computes the size of  spheres of small radius in various Baumslag-Solitar groups (Table 5.1 in \cite{\Tom}).
Using this data and Fekete's Lemma (see page 63 of \cite{\Tom}) he obtains  upper bounds for the spherical growth rates of the following groups. (Note that the spherical growth sequence is submultiplicative in any finitely generated group, so Fekete's Lemma applies.)

\begin{itemize}
\item
For $\bs 22$ the sphere of radius 18 contains 3014654  elements, so an upper bound for the growth rate is $\sqrt[18]{3014654}$ which is approximately $2.290$.
\item
For $\bs 23$ the sphere of radius 18 contains 38595072  elements, so an upper bound for the growth rate is $\sqrt[18]{38595072}$ which is approximately $2.639$.
\item
For $\bs 35$ the sphere of radius 15 contains 11615210 elements, so an upper bound for the growth rate is $\sqrt[15]{11615210}$ which is approximately $2.958$.
\end{itemize}

Recall that if $S(z)$ is the generating function  for the spherical growth series  and $B(z)$ is the generating function for the growth series, then $$B(z)=\frac{S(z)}{1-z}.$$ If the dominant singularity (radius of convergence)  of $B(z)$ is $r$ then the exponential growth rate of the growth series is $\frac1{r}$. Since by Theorem~\ref{thm:growth_rate} the growth rate of $\bs pq$ is bounded below by $2$, the dominant singularity of $B(z)$ is at most $\frac12$,  so the factor $1-z$ in the denominator does not affect the dominant singularity, that is, the spherical growth rate is the same as the growth rate for all $2\leq p\leq q$.

Combining these bounds we have the following estimates.
\begin{itemize}
\item
For $\bs 22$ the  growth rate is between $2$ and $2.290$ (in fact it is exactly 2, see below).
\item
For $\bs 23$ the  growth rate is between $2.302$ and  $2.639$.
\item
For $\bs 35$ the  growth rate is between $2.732$ and $2.958$.
\end{itemize}

In the case $p=q$ the exact growth rates can be obtained from the generating functions obtained by Edjvet and Johnson \cite{\EdjvetJ}. 
For $\bs 22$ the generating function is $$\frac{1-z-2z^3}{(1-z)(1-2z)^2}$$ which has a dominant singularity of $\frac12$, so the growth rate is exactly 2. It follows that the lower bound obtained in Proposition~\ref{prop:22} is sharp. For $\bs 33$, the generating function is $$\frac{(1+z)^2(1-2z)(1+z+2z^3)}{(1-z)(1-z-4z^2)(1-z-2z^2-2z^3)}$$ whose dominant singularity is $0.39039$, giving a growth rate of $2.5616$, the same one we obtain in Theorem~\ref{thm:growth_rate}.
The exact values also agree with our lower bounds for $p=4,5,6$.

\bibliographystyle{plain}
\bibliography{refs}

\begin{thebibliography}{10}

\bibitem{FredenA}
Jared Adams and Eric~M. Freden.
\newblock A context-sensitive combing associated with {B}aumslag-{S}olitar 2,7.
\newblock {\em ACOURAS}, 2:23--38, 2010.

\bibitem{MR1169911}
Marcus Brazil.
\newblock Growth functions for some nonautomatic {B}aumslag-{S}olitar groups.
\newblock {\em Trans. Amer. Math. Soc.}, 342(1):137--154, 1994.

\bibitem{MR1249578}
D.~J. Collins, M.~Edjvet, and C.~P. Gill.
\newblock Growth series for the group {$\langle x,y\vert\
  x^{-1}yx=y^l\rangle$}.
\newblock {\em Arch. Math. (Basel)}, 62(1):1--11, 1994.

\bibitem{MR2787455}
Volker Diekert and J{\"u}rn Laun.
\newblock On computing geodesics in {B}aumslag-{S}olitar groups.
\newblock {\em Internat. J. Algebra Comput.}, 21(1-2):119--145, 2011.

\bibitem{MR1151287}
M.~Edjvet and D.~L. Johnson.
\newblock The growth of certain amalgamated free products and {HNN}-extensions.
\newblock {\em J. Austral. Math. Soc. Ser. A}, 52(3):285--298, 1992.

\bibitem{MR2142503}
Murray Elder.
\newblock A context-free and a 1-counter geodesic language for a
  {B}aumslag-{S}olitar group.
\newblock {\em Theoret. Comput. Sci.}, 339(2-3):344--371, 2005.

\bibitem{MR2776987}
Murray Elder.
\newblock A linear-time algorithm to compute geodesics in solvable
  {B}aumslag-solitar groups.
\newblock {\em Illinois J. Math.}, 54(1):109--128, 2010.

\bibitem{CGA}
Murray Elder and Jennifer Taback.
\newblock $\mathcal c$-graph automatic groups, 2013.
\newblock http://arxiv.org/abs/....

\bibitem{MR2328173}
Eric~M. Freden and Teresa Knudson.
\newblock Recent growth results.
\newblock In {\em Groups {S}t. {A}ndrews 2005. {V}ol. 1}, volume 339 of {\em
  London Math. Soc. Lecture Note Ser.}, pages 341--355. Cambridge Univ. Press,
  Cambridge, 2007.

\bibitem{MR2777003}
Eric~M. Freden, Teresa Knudson, and Jennifer Schofield.
\newblock Growth in {B}aumslag-{S}olitar groups {I}: subgroups and rationality.
\newblock {\em LMS J. Comput. Math.}, 14:34--71, 2011.

\bibitem{MR1425318}
J.~R.~J. Groves.
\newblock Minimal length normal forms for some soluble groups.
\newblock {\em J. Pure Appl. Algebra}, 114(1):51--58, 1996.

\bibitem{KKM}
Olga Kharlampovich, Bakhadyr Khoussainov, and Alexei~G. Miasnikov.
\newblock From automatic structures to automatic groups, 2011.
\newblock http://arxiv.org/abs/1107.3645.

\bibitem{LaunDiss}
J{\"u}rn Laun.
\newblock {\em Solving Algorithmic Problems in {B}aumslag-{S}olitar and their
  Extensions Using Data Compression}.
\newblock PhD thesis, Universit{\"a}t Stuttgart, 2012.

\bibitem{MR0577064}
Roger~C. Lyndon and Paul~E. Schupp.
\newblock {\em Combinatorial group theory}.
\newblock Springer-Verlag, Berlin, 1977.
\newblock Ergebnisse der Mathematik und ihrer Grenzgebiete, Band 89.

\bibitem{TomWong}
Thomas Wong.
\newblock Enumeration problems in {B}aumslag-{S}olitar groups.
\newblock Master's thesis, University of British Columbia, 2010.

\end{thebibliography}

\end{document}